\documentclass{article}

\usepackage{arxiv}

\usepackage[utf8]{inputenc} 
\usepackage[T1]{fontenc}    
\usepackage{hyperref}       
\usepackage{url}            
\usepackage{booktabs}       
\usepackage{amsfonts}       
\usepackage{nicefrac}       
\usepackage{microtype}      
\usepackage{lipsum}

\usepackage{amssymb,amsthm,amsmath,color,makeidx}

\newcommand{\Plm}{(P_1)} 
\newcommand{\intA}{\int_{\R^N}(|\nabla_A u_{\lambda,\mu}|^2+u_{\lambda,\mu}^2)dx} 
\newcommand{\intf}{\int_{\R^N}a_{\lambda}(x)| u|^q dx}
\newcommand{\intg}{\int_{\R^N}b_{\mu}(x)| u|^p dx} 
\newcommand{\jf}{J_{\lambda,\mu}} 
\newcommand{\ume}{u_{\lambda,\mu}^-}
\newcommand{\uma}{u_{\lambda,\mu}^+}

\newcommand{\Ho}{\mbox{$H^1_0 (\R^N)$}}
\newcommand{\HA}{\mbox{$H^1_A (\R^N)$}}
\newcommand{\HAo}{\mbox{$H_A^1(\R^N)\setminus \{0\} $}}
\newcommand{\C}{\mathbb{C}}
\newcommand{\N}{\mathbb{N}}
\newcommand{\R}{{\mathbb R}}

\newtheorem{theorem}{Theorem}[section]

\newtheorem{remark}[theorem]{Remark}
\newtheorem{prop}[theorem]{Proposition}
\newtheorem{lemma}[theorem]{Lemma}

\title{Existence at least four solutions for a Schrödinger equation with magnetic potential involving sign-changing weight function}

\author{
	 de Paiva, Francisco Odair Vieira \thanks{F.O.V.P. received research grants from FAPESP 17/16108-6. } \\
	Departamento de Matem\'atica, UFSCar\\
	S\~{a}o Carlos, SP,  13560-970 Brazil\\
	\texttt{franciscoodair@gmail.com} \\
	\And
	 de Souza Lima, Sandra Machado \thanks{S.M.S.L. was supported by CAPES/Brazil and the paper was completed while the second author was visiting the Departament of Mathematics of UFJF, whose hospitality she gratefully acknowledges.  }\\
	Departamento de Ciências Exatas, Biológicas e da Terra\\ INFES-UFF\\
	Santo Antônio de Pádua - RJ, 28470-000, Brazil\\
	\texttt{sandra.msouzalima@gmail.com} \\
	\And
	Miyagaki, Olimpio Hiroshi  \thanks{ O. H. M. received research grants from CNPq/Brazil 307061/2018-1, FAPEMIG CEX APQ 00063/15 and INCTMAT/CNPQ/Brazil. }\\
	Departamento de Matemática, UFJF\\
	Juiz de Fora, MG, 36036-900, Brazil\\
	\texttt{Corresponding author: {ohmiyagaki@gmail.com}} \\
}

\begin{document}
	\maketitle
	
	\begin{abstract}
		In this paper we consider the following class of elliptic problems
		$$- \Delta_A u + u = a_{\lambda}(x) |u|^{q-2}u+b_{\mu}(x) |u|^{p-2}u ,$$
		for $x \in \R^N$, $1<q<2<p<2^*-1= \frac{N+2}{N-2}$, $a_{\lambda}(x)$ is a sign-changing weight function, $b_{\mu}(x)$ has some aditional conditions, $u \in H^1_A(\R^N)$ and $A:\R^N \rightarrow\R^N$ is a magnetic potential. Exploring the Bahri Li argument and some preliminar results we will discuss the existence of four solution to the problem in question.
	\end{abstract}


	\section{Introduction}
	In this work we are interested in studying the existence of a fourth solution for the following classes of concave-convex elliptical problem
	$$
	\left\{ \begin{array} [c]{ll}
	- \Delta_A u + u = a_{\lambda}(x) |u|^{q-2}u+b_{\mu}(x) |u|^{p-2}u   \, \, \mbox{in} \, \,
	\R^N, & \\
	u \in \HA,&\\
	
	\end{array}
	\right.\leqno {\Plm}
	$$
	\noindent where $N\geq 3 $, $-\Delta_A =(-i\nabla+A)^2$, $1<q<2<p<2^*= \frac{2N}{N-2}$, $a_{\lambda}(x)$ is a family of functions that can change signal, $b_{\mu}(x)$ is continuous and satisfies some additional conditions, $u
	:\R^N \rightarrow \C$ with
	$u \in H^1_A(\R^N)$ (such space will be defined later),  $\lambda > 0$ and $\mu > 0 $ are real parameters, $u:\R^N \rightarrow \C$ and $A:\R^N \rightarrow\R^N$ is a magnetic potential in $L^2_{loc}(\R^N,\R^N)$.
	
	In \cite{doc1} the authors show the existence of three solutions for this problem and also prove their regularity. In this case we will show the existence of fourth solution. 
	Many works have been developed with the magnetic laplaciano. Its importance in physics was discussed e. g. in Alves and Figueiredo \cite{ClauFig} and in Arioli and Szulkin\cite{ArSz}.
	
	There are so many works in literature with  similar problem to $(P_1)$ with $ A = 0 $ like in Ambrosetti, Brezis and Cerami \cite{AmbBre}, where the following problem was considered
	$$
	\left\{ \begin{array}[c]{ll}
	- \Delta  u + u =  \lambda u^{q-1}+  u^{p-1} \, \, \mbox{in} \, \, \Omega, & \\
	u> 0 \, \, \mbox{in}\, \, \Omega,&\\
	u= 0 \, \, \mbox{in}\, \, \partial\Omega,&\\
	\end{array}
	\right.
	$$
	where $\Omega $ is a bounded regular domain of $\R^N$ ($N\geq 3$),  with smooth boundary and $1<q<2<p\leq2^*$.
	Combining the method of sub and super-solutions with the variational method, the authors proved the existence of a certain $\lambda_0>0$ such that there are two solutions when $\lambda\in (0,\lambda_0)$, one solutions if $\lambda=\lambda_0$ and no solutions if $\lambda>\lambda_0$. 
	
	The concave-convex problem like
	$$
	\left\{ \begin{array}[c]{ll}
	- \Delta  u + u = \lambda f(x) u^{q-1}+  u^{p-1} \, \, \mbox{in} \, \, \Omega, & \\
	u> 0 \, \, \mbox{in}\, \, \Omega,&\\
	u= 0 \, \, \mbox{in}\, \, \partial\Omega,&\\
	\end{array}
	\right.
	$$
	with $f\in C(\overline{\Omega})$ a sign changing function and $1<q<2<p\leq2^*$, was studied by Wu in \cite {Wu2}. It proves that the problem has at least two positive solutions for values of $ \lambda $ small enough. Therefore, many studies have been devoted to the analysis of existence and multiplicity of concave-convex elliptic problems in bounded domains, for instance, we can cite Brown \cite{B}; Brown and Wu \cite{BW}; Brown and Zhang \cite{BZ2003}; Hsu \cite{Hsu1}; Hsu and Lin \cite{Hsu} and references contained in these articles.
	
	In an unbounded domain we can cite Chen \cite{Chen}, Huang, Wu and Wu \cite{HWW}, who have worked with a similar cases in $\R^N.$ In \cite{Wu}, Wu deals with the problem
	$$
	\left\{ \begin{array}[c]{ll}
	-\Delta  u + u = f_{\lambda }(x) u^{q-1}+  g_{\mu}u^{p-1} \, \, \mbox{in} \, \, \R^N, & \\
	u\geq 0 \, \, \mbox{in}\, \, \R^N,&\\
	u \,\, \in \,\, H^1(\R^N),&\\
	\end{array}
	\right.
	$$
	with $1<q<2<p\leq2^*$, $g_{\mu}\geq 0$ or $f_{\lambda}$ being able to change of signal, among other additional hypotheses. It seeks to show the existence of at least four solutions to the problem when $\lambda$ and $\mu$ small enough.
	This result was extend in \cite{doc1}, investigating if it would be possible to obtain similar consequences when we replace the magnetic laplacian in the place of the usual Laplacian. In this work we will show the existence of the fourth solution for this problem.

	The first results in non-linear Schr\"{o}dinger equations, with $ A \neq 0 $ can be atributed to Esteban and Lions \cite{EstLions} in which the existence of stationary solutions for equations of the type
	$$-\Delta_A+Vu =|u|^{p-2}u, u\neq 0, u\in L^2(\R^N)  , $$
	with $ V = 1 $ and $p \in (2,\infty),$ were obtained using minimization method with constant magnetic field and also for the general case.
	
	Chabrowski and Szulkin \cite{ChabSzul} worked with this operator in the critical case and with the electric potential V being able to change the signal. Already Cingolani, Jeanjean and Secchi \cite{Cingolani2} consider the existence of mult-peak solutions in the subcritical case.
	
	A problem of the type
	$$-\Delta_A u = \mu |u|^{q-2}u+|u|^{2^*-2}u, u \neq  0,  \Omega \subset \R^N ,  $$
	with $\mu>0$ and $2\leq q <2^*$, is treated by Alves and Figueiredo \cite{ClauFig} in which the number of solutions with the topology of $ \Omega $ is related.

	In \cite{doc1} they deal with the non-zero $ A $ case with a weight function that changes sign in the concave-convex case, like the problem in this work. They prove the existence of three solutions for the problem and now, we would like to show the existence of the fourth solution. In \cite{doc1} was used the Nehari manifold linked with the behavior of functions known as fibering map and Category theory.

	In the sequence we will announce some preliminars results and the  result that we seek to show. Observe that
	\begin{equation}\label{funcional.1}
	\jf(u)=\frac{1}{2}\intA - \frac{1}{q} \intf - \frac{1}{p} \intg,
	\end{equation}
	is the functional associated with the problem $ \Plm $ and is of class $ C^1 $ in $ \HA $ as can be seen in \cite{Rab}. Also, the critical points of $\jf(u)$ are weak solutions of problem $(P_1)$. 
	We will work with the hypotheses that we will enunciate next. Consider the function $a(x) \in L^{q'}(\R^N),\; q'=\frac{p}{p-q}$ and $a_{\pm} =\pm \max\{\pm a(x),0\} \neq 0$. Let us assume
	$$a_{\lambda} (x)=\lambda a_+(x)+a_-(x).$$ 
	\begin{description}
		\item[  $(A)$ ] $  a(x) \in L^{q'}(\R^N),\; q'=\frac{p}{p-q}$ and exists $ \hat{c}>0$ and $r_{a_-}>0, $ such that
		$$a_-(x)>-\hat{c} \exp(- r_{a_-}|x|) \;\;\mbox{ for all  }\;\; x \in  \,\R^N. $$   
	\end{description}
	In addition to $(A)$, we will assume that $b_{\mu}(x)=b_1(x)+\mu b_2 (x)$, where 
	\begin{description}
		\item[ $(B_1)$  ]$ b_1(x)>0$ in continuous in $ \R^N $, with $b_1(x)\rightarrow 1$ as $|x|\rightarrow \infty$ and exists $r_{b_1}>0$, such that
		$$1\geq b_1(x) \geq 1-c_0\exp(-r_{b_1}|x|) \;\; \mbox{for some }\;\; c_0<1 \;\; \mbox{and for all } \;\; x \in \R^N.$$   
		\item[   $(B_2)$]    $  b_2(x)>0$ is continuous in $ \R^N $, $b_2(x)\rightarrow 0$ as $|x|\rightarrow \infty$ and exists $r_{b_2}>0$, with
		$r_{b_2}<\min\{  r_{a_-}, r_{b_1},q\}$ such that
		$$b_2(x)\geq d_0 \exp(-r_{b_2}|x|) \;\; \mbox{for some }\;\; d_0<1 \;\; \mbox{and for all } \;\; x \in
		\R^N.$$
	\end{description}
	Those hypotheses were used in \cite{doc1}. Consider
	$$\Upsilon_0=( 2-q )^{ 2-q } \left(  \frac{p-2}{ ||a_+||_{q'} }\right)^{p-2}   \left(  \frac{S_p
	}{p-q }\right)^{p-q}, \;\;\; \mbox{where} $$
	\begin{equation}\label{Sp} 
	S_p=\inf_{u\in H_A^1(\R^N \setminus \{0\})}
	\frac{\left(\int_{\R^N}|\nabla_Au|^2+u^2dx\right)^{\frac{1}{2}}}{\left(\int_{\R^N}|u|^pdx\right)^{\frac{2}{p}}}>0.
	\end{equation}

	In \cite{doc1} the first result, assuming the hypotheses $ (A), \; (B_1) $ and $ (B_2) $, and $ \Upsilon_0 $ as defined above, it was proved that $\Plm$ has at least one solution, provided that 
	\begin{equation}\label{des.do.teo}
	\lambda^{p-2}(1+\mu||b_2||_{\infty})^{2-q}< \left(\frac{q}{2}\right)^{p-2} \Upsilon_0
	\end{equation}
	holds for each $\lambda >0$ and $\mu>0$. Then, adding the hypothesis that the potential is asymptotic to a constant in infinity, they prove the existence of at least two solutions $\uma$ and $\ume$ with $\jf(\uma)<0<\jf(\ume).$

	In the previous result, the existence is valid for all  $ \lambda $ and $ \mu $ satisfying the inequality (\ref{des.do.teo}). So, if we additionally set values of $ \lambda $ and $ \mu $ conveniently small we obtain the multiplicity result, that is, the existence of at least three solutions. Actually they showed the existence of $\lambda_0 >0$ and $\mu_0>0$ with $\lambda_0^{p-2}(1+\mu_0||b_2||_{\infty})^{2-q}< \left(\frac{q}{2}\right)^{p-2}\Upsilon_0$, such that for all  $\lambda \in (0,\lambda_0)$ and $\mu \in (0,\mu_0)$, the problem  $\Plm$ has at least three solutions.
	
	Now, in this work, we observe that for the problem in question, the numbers $ \lambda_0 $ and $ \mu_0 $ as previously mentioned are independent of the value of $ a_-. $ However, considering some additional hypotheses and taking values of $||a_-||_{q'}$ sufficiently small we have the results getting another solution. Before enunciate this result we will present the following hypotheses:
	
	\begin{description}
		\item[ $(C_1)$  ] $b_1(x)< 1$ in $\R^N$ in a positive measure set;  
		\item[ $(C_2)$  ] $r_{b_1}>2.$  
	\end{description}

	\begin{theorem}\label{teo1.2} 	Suppose that the potential $A \rightarrow d$ where $d$ constant as $|x|\rightarrow \infty$. Assuming the hypotheses $(A)$, $(B_1)$, $(B_2)$, $(C_1)$ and $(C_2)$ there are positive values of $\tilde{\lambda_0}\leq \lambda_0$, $\tilde{\mu_0}\leq \mu_0$ and $ \nu_0$ such that for $\lambda \in (0, \tilde{\lambda_0})$, $\mu \in (0, \tilde{\mu_0})$ and $||a_-||_{q'} < \nu_0$,  the problem  $\Plm$ has at least four solutions. \end{theorem}
	
	For these first three solutions results of this problem was used the Nehari method together with the category theory. 
	We will continue to make use of variational methods to prove the above theorem. We will work under a few more assumptions to estimate different energy levels and will use the Bahri-Li min-max argument to show that for very small values of  $||a_-||_{q'}, $ the problem has at least four distinct solutions.

	\section{ Initial considerations}
	According to Tang \cite{TZ}, we denote by $ H_A (\R^N)$ the Hilbert space obtained by the closing of $ C_0^{\infty} (\R^N, \C) $ with following inner product:
	$$<u,v>_A=Re\int_{\R} (\nabla_Au \overline{\nabla_Av}+u\overline{v} dx ),$$
	where $\nabla_Au:=(D_1u,D_2u,...,D_Nu)$ and $D_j:=-i\partial_j-A_j(x) $, with $j=1,2,...,N$, with $A(x)=(A_1(x),...,A_N(x))$. The norm induced by this product is given by
	$$ ||u||_A^2:= \int_{\R}(|\nabla_Au|^2+u^2dx) .$$
	It is proved by Esteban and Lions, \cite[Section II]{EstLions} that for all $u \in H^1_A (\R^N)$ it is worth diamagnetic inequality
	$$|\nabla|u|(x)|=\left|Re\left(\nabla u \frac{\overline{u}}{|u|}\right)\right|=\left|Re\left((\nabla u-iAu)\frac{\overline{u}}{|u|}\right)\right|\leq |\nabla _Au(x)|$$
	So, if $u \in \HA $ we have that $|u|$ belongs to the usual Sobolev space $\Ho$.
	
	\subsection{Preliminar results}
	
	To obtain results of existence in this case, we introduced the Nehari manifold
	$$M_{\lambda, \mu}=\{u \in \HA\setminus\{0\} :\langle J'_{\lambda,\mu}(u),u\rangle =0\},$$
	where $\langle \;\;,\;\;\rangle$ denotes the usual duality between $\HA^*$ and $\HA$, where $\HA^*$ is the dual space to the corresponding $ \HA $ space.
	The Nehari manifold is linked to the functions of the form $F_u:t\rightarrow \jf(tu);\;\;(t>0)$, called fibering map. Note that the fabering map it was defined and depends on $ u $, $ \lambda $ and $ \mu $, so that proper notation would be $ F_{u, \lambda, \mu} $, but in order to simplify the notation, we will denote by $ F_u $.
	If $u \in  \HA$, we have
	\begin{equation}\label{phi}
	F _u(t)=\frac{t^2}{2} ||u||_A^2  -\frac{t^{q}}{q}\intf - \frac{t^{p}}{p}\intg,
	\end{equation}
	\begin{equation}\label{phi'}
	F' _u(t)= t ||u||_A^2 -  t^{q-1} \intf -  t^{p-1} \intg,
	\end{equation}
	\begin{equation}\label{phi''}
	F'' _u(t)= ||u||_A^2 - (q-1)t^{q-2}\intf -(p-1)t^{p-2}\intg.
	\end{equation}
	The following remark relates the Nehari manifold and the Fibering map.
	
	\begin{remark}\label{u_in_s}
		Let $ F_u $ be the application defined above and $ u \in \HA $, then:
		
		$(i)$ $ u \in M_{\lambda, \mu}$ if, and only if, $F_u'(1)=0$;
		
		$(ii)$ more generally $ tu \in M_{ \lambda, \mu} $, and only if, $F'_u (t)=0$.
	\end{remark}

	From the previous remark we can conclude that the elements in $ M_{\lambda, \mu}$, correspond to the critical points of the Fibering map. Thus, as $ F_u (t) \in C^2(\R^+, \R) $, we can divide the Nehari manifold into three parts
	$$ M_{\lambda, \mu}^+=\{ u \in M_{\lambda, \mu}; F''_{\lambda, \mu}(1)>0  \}; $$
	$$ M_{\lambda, \mu}^-=\{ u \in M_{\lambda, \mu}; F''_{\lambda, \mu}(1)<0  \}; $$
	$$ M_{\lambda, \mu}^0=\{ u \in M_{\lambda, \mu}; F''_{\lambda, \mu}(1)=0  \}. $$

	Lemma below shows us under some conditions the $M^0_{\lambda, \mu}$ is empty.
	
	\begin{lemma}\label{N0}
		Let $\mu\geq 0$ and $\lambda>0$ such that
		\begin{equation}\label{lambda}
		\lambda^{p-2}(1+\mu||b_2||_{\infty})^{2-q} < \Upsilon_0.
		\end{equation}
		Then $M^0_{\lambda, \mu}=\emptyset$.
	\end{lemma}
	\begin{proof}The proof is similar to what was done in \cite[Lemma 2.2]{BW}.
		
	\end{proof}
	
	In \cite{doc1} they showed that under certain conditions on $\lambda$ and $\mu$, we have a minimizer in $M^+_{\lambda,\mu}$ and another in $M^-_{\lambda,\mu}$, whose minimum levels of energy will be denoted respectively by
	$$m^+_{\lambda,\mu}=\inf_{u \in M^+_{\lambda,\mu}}\jf (u) $$
	and
	$$m^-_{\lambda,\mu} =\inf_{u \in M^-_{\lambda,\mu}}\jf (u)  .$$
	
	To establish the existence of the first two solutions and compare with the energy level of the fourth solution, we will need the following result that was shown in \cite{doc1}.
	
	\begin{lemma}\label{Nehari}
		For each $u \in  \HA \setminus\{0\}$ and $\mu >0$ we have
		
		\begin{description}
			\item[ 	$(i)$  ]    If $\intf \leq 0$, there is a single $t^-(u) > t_{\max}(u)$ such that $t^-(u)u \in M^-_{\lambda,\mu}$. Also, $F_{u}(t)$ is increasing in $(0,t^-(u))$, decreasing in $(t^-(u),+\infty)$ and $F_{u}(t)\rightarrow -\infty$ as $t\rightarrow +\infty$.
			\item[ $(ii)$  ]   If $\intf > 0$ and $\lambda$ is such that $\lambda^{p-2}(1+\mu||b_2||_{\infty})^{2-q}<\Upsilon_0$, so there is $0<t^+(u) < t_{\max}(u)<t^-(u) $ such that $t^\pm(u)u \in M^\pm_{\lambda,\mu}$. Also, $F_{u}(t)$ is decreasing in $(0,t^+(u))$, increasing in $(t^+(u),t^-(u))$ and decreasing in $(t^-(u),+\infty)$. Furthermore, $F_{u}(t)\rightarrow -\infty$ as $t\rightarrow +\infty$.
		\end{description}
	\end{lemma}

	Our next result shows that these points are well defined.
	
	\begin{lemma}\label{bounded}
		The functional $\jf$ is coercive and bounded from below in $M_{\lambda,\mu}$.
	\end{lemma}
	
	\begin{proof} The proof is similar to that made in \cite[Lemma 2.1]{Hsu}.
	\end{proof}
	
	For the next results we will need some estimates about the values of the functions in $m^\pm_{\lambda,\mu}$. To do this, from (\ref{lambda}) we have
	\begin{equation*} 
	||u||_A^2<   \frac{p-q}{p-2}\intf  \leq  \Upsilon_0^{1/(p-2)} \frac{p-q}{p-2} S_p^{\frac{-q}{2}} ||a_+||_{L^{q'}} ||u||_A^q.
	\end{equation*}
	Therefore
	\begin{equation}\label{u.em.nmais}
	||u||_A \leq \left( \Upsilon_0^{1/(p-2)} \frac{p-q}{p-2} S_p^{\frac{-q}{2}} ||a_+||_{L^{q'}}\right)^{1/(2-q)} ||u||_A^q,
	\end{equation}
	for all $u\in M^+_{\lambda,\mu}.$ 
	Also, if $\lambda=0$, then (\ref{lambda}) is satisfied, so that by Lemma \ref{Nehari}(i), $ M^+_{\lambda,\mu}=\emptyset $, and we have $ M_{\lambda,\mu}= M^-_{\lambda,\mu}$ for all $\mu\geq 0$.
	By has been seen, we will show the following results on the values of $m^\pm_{\lambda,\mu}$. 
	
	\begin{lemma}\label{teorema3.1}
		\begin{description}
			\item[ $(i)$  ] If $\lambda^{p-2} (1+\mu||b_2||_{\infty})^{2-q} < (\frac{q}{2})^{p-2} \Upsilon_0, $ then $m^-_{\lambda,\mu}>0$; 
			\item[$(ii)$   ] For $\lambda>0$ and $\mu \geq 0$ with $\lambda^{p-2}(1+\mu||b_2||_{\infty})^{2-q}<\Upsilon_0,$ then
			$m^+_{\lambda,\mu}<0.$ 
			In particular, if $\lambda^{p-2}(1+\mu||b_2||_{\infty})^{2-q}<(\frac{q}{2})^{p-2}\Upsilon_0,$ then
			$$m^+_{\lambda,\mu}=\inf_{M_{\lambda,\mu}}\jf(u).$$ 
		\end{description}
	\end{lemma}
	
	\begin{proof} The proof is similar to what was done in \cite[Theorem 3.1]{Wu}.
	\end{proof}
	
	By Lemma \ref{teorema3.1}, we can conclude that for every $u \in \HAo$
	\begin{equation}
	\jf (t^-(u)u) = \max_{t\leq 0} \jf(tu),
	\end{equation}
	whenever $ \lambda^{p-2} (1+\mu||b_2||_{\infty})^{2-q}  < \left( \frac{q}{2} \right)^{p-2}\Upsilon_0,$ with $\lambda\geq 0 $ and $\mu>0.$

	\section{Existence of $m_{\infty}$ }\label{existencia}

	In this section we will define the energy level of the limit problem and make some energy estimates in relation to the energy levels of the solutions in the Nehary manifold. Therefore, we will have tools to show that the fourth solution to be found has a different level than other solutions already found. For this, consider the following semilinear elliptical problem
	$$
	\left\{ \begin{array} [c]{ll}
	- \Delta_A  u + u =  |u|^{p-2}u   \, \, \mbox{in} \, \,            \R^N, & \\
	u \in \HA.&\\
	\end{array}
	\right.\leqno {(P_A)}
	$$
	
	Define $J_{\infty}(u)=\frac{1}{2}||u||^2_A-\frac{1}{p}||u||_p^p$, the functional associated with the problem ($P_A$), then $J_{\infty}$ is a functional $C^2$ in $\HA$. The Nehari manifold associated with problem $(P_A)$ is given by 
	$$M_{\infty}=\{ u \in \HAo; \; J'_{\infty}(u)u=0 \}.$$ 
	In this problem we can observe that if $ u \in N_{\infty}$, then $||u||_A^2=||u||_p^p$.
	Now consider the following minimization problem
	\begin{equation}\label{m.inf}
	m_{\infty}= \inf_{M_{\infty}}J_{\infty}(u).
	\end{equation}
	In \cite{doc1} they prove that exists $\bar{u} \in \HA$ such that $m_{\infty}= \inf_{N_{\infty}}J_{\infty}(u)=J_{\infty}(\bar{u})$. 
	From these considerations we will show the following result that gives us a description of a sequence (PS) of $J_{\lambda,\mu}$.
	
	\begin{lemma}\label{Lemmadecompacidade}
		
		Let $\{u_n\}\subset M^-_{\lambda, \mu}$ be a sequence $(PS)_{\beta}$ in $\HA$ of $J_{\lambda,\mu} $, this is, a sequence satisfying $J_{\lambda,\mu}(u_n)=\beta +o_n(1)$  and $J'_{\lambda,\mu}(u_n)= o_n(1)$ in $H^{-1}_A$ as $n\rightarrow \infty$, where
		$$m^+_{\lambda,\mu}+ m_{\infty}<\beta <m^-_{\lambda,\mu}+ m_{\infty},$$
		then there is a subsequence $\{u_n\}$ and $u_0 \in \HA$, with a non zero $u_0$, such that $u_n=u_0+o_n(1)$ strong in $\HA$ and $J_{\lambda,\mu}(u_0)=\beta $. Moreover, $u_0$ is a solution of $\Plm.$
	\end{lemma}
	
	\begin{proof}
		For $(A), (B_1)$ and $ (B_2)$, we obtain by a standard argument that $\{u_n\}$ is bounded sequence in $\HA$. Then there is a subsequence $\{u_n\}$ and $u_0 \in \HA$ such that $u_n  \rightharpoonup u_0$ weakly in $\HA$ as $n\rightarrow \infty$. Taking $v_n = u_n-  u_0$, we have $v_n \rightharpoonup 0$ weak in $\HA$ as $n\rightarrow \infty$. 
		Denoting by $B(0,1)$ the ball centered on the origin of radius 1, we have in $ B (0,1) $ the strong convergence
		$$ \int_{B(0,1)} |u_n|^q \rightarrow \int_{B(0,1)} |u_0|^q.$$
		By the Dominated Convergence Theorem we obtain
		$$\int_{B(0,1)}a_{\lambda} ||u_n|^q-|u_0|^q| \rightarrow 0,\;\; \mbox{when}\;\;n\rightarrow \infty.$$
		Then, by H\"{o}lder and the integrability of $a_{\lambda}$ follows
		{\small
			\begin{eqnarray*}
				\left|\int a_{\lambda} (x)(|u_n|^q-|u_0|^q)\right|&\leq& o_n(1) + \int_{ B^c(0,1)}a_{\lambda}(x)||u_n|^q-|u_0|^q|\\
				&\leq &  o_n(1) +\left( \int_{ B^c(0,1)}a_{\lambda}(x) ^{q^*}\right)^{\frac{1}{q^*}}( ||u_n||^q_p+||u_0||_p^q)\\
				&\leq& o_n(1)+ \epsilon C .
		\end{eqnarray*}}	
		As $\epsilon >0$ it is arbitrary, we have
		$$ \int a_{\lambda} (x)(|u_n|^q-|u_0|^q)=o_n(1).$$
		
		On the other hand, $(B_1)$ and $(B_2)$ and by Brezis-Lieb lemma (see \cite{willem}), we can conclude that $\mu\int b_2(x) |v_n|^p=o_n(1)$, $\int (1-b_1(x))|v_n|^p=o_n(1)$ and $\int b_{\mu}(x) (|u_n|^p - |v_n|^p-|u_0|^p)=o_n(1)$, which together with the above inequality gives us
		$$  J_{\lambda , \mu}(u_n)= J_{\infty}(v_n)+ J_{\lambda , \mu}(u_0) +o_n(1). $$
		In a similar way we obtain that $ J'_{\infty}(v_n)v_n= J'_{\lambda , \mu}(u_n)u_n- J'_{\lambda , \mu}(u_0)u_0 +o_n(1) $. By hypothesis $J'_{\lambda , \mu}(u_n) \rightarrow 0 $ strong in $\HA^{-1}$ and $u_n \rightharpoonup u_0$ weak in $\HA$ as $n \rightarrow \infty$ and so we have $J'_{\lambda , \mu}(u_0)=0$.
		Now, define $\delta= \limsup_{n\rightarrow \infty} \sup_{y\in \R^N} \int_{B(y,1)}|v_n|^p.$ So we have two cases:
		\begin{description}
			\item[ $(i)$ ] $\delta >0$, or  
			\item[$(ii)$   ]$  \delta = 0$.   
		\end{description}
		
		Suppose that $(i)$ happen. Then there will be a sequence $\{y_n\}\subset \R^N$ such that $ \int_{B(y_n,1)} |v_n|^p\geq \frac{\delta}{2}$ and for all $n \in \N$.
		Define $\tilde{v}_n (x)=v_n(x+y_n)$. We have that $\{\tilde{v}_n  \}$ is bounded and $\tilde{v}_n  \rightharpoonup v$ weak and almost everywhere.
		Making a change of variables we obtain
		$$ \int_{B(0,1)}|\tilde{v}_n |^p \geq \frac{\delta }{4}.$$
		Then
		\begin{equation}\label{v.maior.q.zero} 
		\int_{B(0,1)}|v |^p \geq \frac{\delta }{4},
		\end{equation}
		giving us $v\neq 0$.
		But, $ v_n \rightharpoonup 0$ weakly, then  
		\begin{equation} \label{vn}
		\int_{\R^N}| v_n|^p \geq\int_{B(y_n,1)}|v_n|^p \geq \frac{\delta }{2}>0.
		\end{equation} 
		See that
		$$J_{\infty}(v_n)=\frac{1}{2}\int(|\nabla_Av_n|^2+v_n^2)dx-\frac{1}{p}\int|v_n|^p dx.$$
		Likewise,
		$$F_{v_n}(t)=J_{\infty}(tv_n)=\frac{t^2}{2}||v_n||_A^2-\frac{t^p}{p}||v_n||^p.$$
		For each $n \in \N$, we can get $t_n$ such that $t_nv_n \in M_{\infty}$. So we build a sequence $\{t_n\}\subset \R^N$ with $t_n\rightarrow t_0$ as $n\rightarrow \infty$, such that $t_nv_n \in M_{\infty}$, that is, such that $J'_{\infty}(t_nv_n)t_nv_n=0$. 
		See also that
		$$J'_{\infty}(v_n)v_n=||v_n||_A^2-||v_n||^p=o_n(1)$$
		and	
		\begin{equation}\label{a.i} F'_{v_n}(t)=J'_{\infty}(tv_n)v_n=t||v_n||_A^2-t^{p-1}||v_n||^p=o_n(1).
		\end{equation} 
		With this
		\begin{equation}\label{vai.p.zero} 
		(t_n-t_n^{p-1})||v_n||_A^2=t_n(1-t_n^{p-2})||v_n||_A^2=o_n(1). 
		\end{equation}
		For (\ref{v.maior.q.zero}) we know that $||v_n||^2_A\nrightarrow 0$ (that is, $v_n$ does not converge to zero). Also note that
		$t_n^{2-p}=\frac{\int|v_n|^{p}}{||v_n||_A^2}\geq \frac{\delta}{2c}  .$
		With that and by (\ref{vai.p.zero}) we get that 
		$(1-t_n^{p-2})\rightarrow 0 $, giving us that $t_n\rightarrow 1.$
		Now, see that $v_n \rightharpoonup 0$ weak in $\HA$ as $n\rightarrow \infty$. With this and by the fact $t_n \rightarrow 1$, we can conclude that
		$$ J_{\lambda, \mu}(u_n)=J_{\infty}(t_nv_n) +J_{\lambda, \mu}(u_0)+ o_n(1)\geq m_{\infty}+J_{\lambda, \mu}(u_0) .   $$
		Note that by hypothesis $ J_{\lambda, \mu}(u_n)=\beta+ o_n(1)$ with $\beta <m_{\infty} +m_{\lambda, \mu}^+$. From there we obtain
		$$\beta+ o_n(1)= J_{\lambda, \mu}(u_n)=J_{\infty}(t_n v_n) +J_{\lambda, \mu}(u_0)+ o_n(1)\geq m_{\infty}+J_{\lambda, \mu}(u_0)  ,$$
		giving us
		$$m_{\infty}+J_{\lambda, \mu}(u_0)  \leq  \beta + o_n(1) <m_{\infty} +m_{\lambda, \mu}^+ + o_n(1),$$
		therefore
		\begin{equation}\label{estm^+.}
		J_{\lambda, \mu}(u_0)  < m_{\lambda, \mu}^+ + o_n(1).
		\end{equation}
		We have already seen that $J'_{\lambda, \mu}(u_n)$ converges strongly to zero, therefore we get $J'_{\lambda, \mu}(u_0)=0$. Thus $u_0 \in M_{\lambda, \mu}.$ Still, by Lemma \ref{N0}, $M^0_{\lambda, \mu} =  \emptyset$ and by Lemma \ref{teorema3.1}, we conclude that $m^+>0 $ and $m^-<0.$ Then,
		$$J_{\lambda, \mu}(u_0)\geq \inf_{M_{\lambda, \mu}}J_{\lambda, \mu}(u) = \inf_{M^+_{\lambda, \mu}}J_{\lambda, \mu}(u)=m^+,$$
		\noindent which contradicts what we have concluded in (\ref{estm^+.}).
		We conclude that ($ii$) occurs. In this case, $\{v_n\}$ such that $\int|v_n|^p\rightarrow 0$ if $n\rightarrow \infty$.
		
		As we already have $J'_{\infty}(v_n)v_n=o_n(1)$ with $J'_{\infty}(v_n)v_n=||v_n||_A^2-||v_n||_p^p$ and $\int|v_n|^p\rightarrow 0$, we conclude that $||v_n||^2\rightarrow 0$ giving us $u_n \rightarrow u_0$ strong in $\HA$.
		See also that $u_0 \neq 0$. In fact, note that if $u_0=0$ so $\tilde{v}_n =v_n=u_n$ and $ \int_{B(0,1)}|u_n|^p \geq \frac{\delta }{4}$, which we have already seen to be an absurd.

	\end{proof}

	
	To treat the existence of the second solution of the problem $ \Plm $, we need to make some considerations. Note that equation
	$$- \Delta_A u + u = a_{\lambda}(x) |u|^{q-2}u+b_{\mu}(x) |u|^{p-2}u  \;\;\;\;\;\; \Plm $$
	is such that $a_{\lambda}(x) \rightarrow 0 $ and $ b_{\mu}(x) \rightarrow 1$ as $|x| \rightarrow \infty$. Adding the hypothesis of $ A \rightarrow d$ with $d$ constant as $|x| \rightarrow \infty$, the problem  $\Plm$ converges at infinity for the problem
	$$- \Delta_d u + u = |u|^{p-2}u.  \;\;\;\;\;\;(P_{\infty}), $$
	where $-\Delta_d=(-i\nabla+d)^2$.
	Thus, by a result of Ding and Liu \cite[Lemma 2.5]{DL}, $u$ is a solution of Problem $(P_{\infty})$ if and only if $v(x):=|u(x)| \in H^1$ it is a solution to the problem
	$$- \Delta v + v = v^{p-1}; \;\;v>0 .  \;\;\;\;\;\; (E_{\infty}) $$
	Moreover, the equations $(P_{\infty})$ and $ (E_{\infty}) $ have the same energy level, that is
	$$ J_{\infty} (u )  = I_{\infty} (v ) = m_{\infty};$$
	on what $ J_{\infty}$ and $I_{\infty}$ are the respective functional associated with the previous problems.
	Acording to Berestick, Lions \cite{BerLions} or Kwong \cite{KWONG}, the equation $(E_{\infty}) $ has a unique solution $z_0$ symmetrical, positive and radial. By \cite[Theorem 2]{GNN}, for all  $\epsilon >0,$ exists $A_{\epsilon}, B_0$ and $C_{\epsilon}$ positive such that
	\begin{equation}\label{exp1}
	A_{\epsilon} \exp (-(1+ \epsilon)|x|)\leq z_0(x) \leq B_0 \exp (-|x|)
	\end{equation}
	and
	\begin{equation}\label{exp2}
	|\nabla z_0(x)| \leq C_{\epsilon} \exp (-(1- \epsilon)|x|).
	\end{equation}
	According Kurata \cite[Lemma 4]{Kurata}, defining $w_0=z_0 e^{-idx}$ we have $w_0$ is a solution of $ (P_{\infty}) $, unique, symmetrical, positive and radial. So we will have $ J_{\infty} (w_0 )  =m_{\infty}$. See also that $z_0=|w_0|$, which together with (\ref{exp1}) gives us the following inequalities
	\begin{equation}\label{exp1w}
	A_{\epsilon} \exp (-(1+ \epsilon)|x|)\leq |w_0(x)| \leq B_0 \exp (-|x|)
	\end{equation}
	and
	\begin{equation}\label{exp2w}
	|\nabla w_0(x)| \leq C_{\epsilon} \exp (-(1- \epsilon)|x|).
	\end{equation}
	
	Next, we will make some estimates about the minimum energy levels in the Nehari Manifold to prove the existence of a second solution.
	In order to not overload the notation, we will denote $ u_{\lambda,\mu}^+:= u^+$. Considering $ J(u^+) =m^+$, $ m^- = \inf_{u  \in M^-_{\lambda,\mu }}J_{\lambda,\mu }(u)$ and $ m_{\infty} = \inf_{u  \in M_{\infty}}J_{\infty}(u)= J_{\infty} (w_0 )$, we will make the following estimate for such energy levels.

	\begin{prop}\label{prop4.1}
		For all  $\lambda>0$ and $\mu> 0$ satisfying $ \lambda^{p-2} (1+\mu||b_2||_{\infty})^{2-q}< \Upsilon_0$, we have $m^- < m^+ + m^{\infty}$.
	\end{prop}
	
	\begin{proof}The proof is similar to what was done in \cite[Proposition 6.1]{doc1}.
		
	\end{proof}

	\section{ Third Solution}
	
	\subsection{ Some considerations}
	To get the third solution of the $ \Plm $ problem, we will need some results that is done next. For this, we highlight the set defined below for $\lambda=0$ and $\mu=0$
	$$ M^-_{a_0,b_0}=\{u \in \HA\setminus\{0\} :\langle J'_{a_0,b_0}(u),u\rangle =0\}$$
	where 
	\begin{eqnarray*}\label{funcional.0}
		J_{a_0,b_0}&=&\frac{1}{2}\intA - \frac{1}{q} \int a_0(x)|u|^qdx - \frac{1}{p} \int b_0(x)|u|^pdx\\
		&=&\frac{1}{2}\intA - \frac{1}{q} \int a_-(x)|u|^qdx - \frac{1}{p} \int b_1(x)|u|^pdx.
	\end{eqnarray*}
	
	\begin{lemma}\label{5.1.a}
		We have
		$$\inf_{u \in M^-_{a_0,b_0}}J_{a_0,b_0}(u)=\inf_{u \in M^{\infty}}J_{\infty}(u)=m^{\infty}.$$
	\end{lemma}
	
	\begin{proof}
		Let $w_k$ be as defined above. Because of we are working with $ \lambda = 0 $, we have $a(x)=\lambda a_+(x)+a_-(x)=a_-(x)<0$ from where $\int_{\R^N} a_-| t^-(w_k)w_k|^qdx\leq 0,$ hence by Lemma \ref{Nehari}(i) there is only one $t^-(w_k)>\left(\frac{2-q}{p-q}\right)^{\frac{1}{p-2}}$ such that $t^-(w_k)w_k \in  M^-_{a_0,b_0}$ for all $k>0,$ that is, $ J'_{a_0,b_0} (t^-(w_k)w_k)=0, $ giving us
		\begin{equation}\label{t_menos}
		||t^-(w_k)w_k||^2_A= \int_{\R^N} a_-| t^-(w_k)w_k|^qdx +\int_{\R^N} b_-| t^-(w_k)w_k|^p dx.
		\end{equation}
		As $w_0$ is a solution of problem ($E_{\infty}$) and remembering that the functional associated with ($E_{\infty}$) is given by $I(u)=\frac{1}{2}||u||_A^2-\frac{1}{p}||u||_p^p,$ and $I'(u)=||u||_A^2- ||u||_p^p$ we have
		$$I'(w_0)w_0=||w_0||_A^2- ||w_0||_p^p=0.$$
		Therefore
		\begin{eqnarray*}
			\nonumber  m_{\infty} &=& I(w_0)=\frac{1}{2}||w_0||_A^2-\frac{1}{p}||w_0||_p^p \\
			\nonumber    &=&  \frac{1}{2}||w_0||_A^2-\frac{1}{p}||w_0||_A^2=\frac{p-2}{2p}||w_0||_A^.
		\end{eqnarray*}
		Being $w_0$ solution of problem ($E_{\infty}$) follows that $w_k(x)=w_0(x+ke)$. With this and $ I'(w_0)w_0=0,$ we have $I'(w_k)w_k=0$. So that
		\begin{equation}\label{m_infty>0}
		|| w_k||^2_A= \int_{\R^N}  |  w_k|^qdx =\frac{2p}{p-2}m^{\infty} \;\; \mbox{for all}\;\;k\geq 0.
		\end{equation}	
		It is known that $ w_n $ is bounded in $L^{r'}$ and $ w_n \rightarrow 0 $ a.e., by Theorem \cite[Theorem 13.44]{Hewitt} that $ w_n \rightharpoonup 0 $ weakly in $ L^{r'}$. By the condition ($A$), $ a_- \in (L^{r'})'=L^r$ we get
		\begin{equation}\label{int_a}
		\int_{\R^N} a_- |w_k|^qdx \rightarrow 0 \;\; \mbox{as}\;\;k\rightarrow \infty.
		\end{equation}
		In addition, by $(B_1)$ and $(B_2)$ we get
		\begin{equation}\label{int_b}
		\int_{\R^N}(1-b_1) | w_k|^qdx =  \int_{B(0,R)}(1-b_1) | w_k|^qdx+ \int_{B^c(0,R)}(1-b_1) | w_k|^qdx \rightarrow 0,
		\end{equation}
		as $|w_k| \rightarrow \infty$.
		By (\ref{t_menos}), (\ref{int_a}) and (\ref{int_b}) we have that $t^-(w_k)\rightarrow 1 $ as $ k\rightarrow \infty.$ 
		Likewise
		$$ \lim_{k\rightarrow \infty} J_{a_0,b_0} (t^-(w_k)w_k)= \lim_{k\rightarrow \infty} J_{\infty} (t^-(w_k)w_k)=m_{\infty}.$$
		Thus
		\begin{eqnarray}\label{geq}
		m_{\infty} &=& \inf_{u \in M^{\infty}}J_{\infty}(u)=\lim_{k\rightarrow \infty} J_{\infty} (t^-(w_k)w_k)\geq \inf_{u \in M^-_{a_0,b_0}}J_{a_0,b_0}(u).
		\end{eqnarray}
		We also have to $u\in  M_{a_0,b_0}$, by Lemma \ref{Nehari}(i), $J_{a_0,b_0}(u)=\sup_{ t\geq 0}J_{a_0,b_0}(tu),$ and more, there is a single $t^{\infty}>0$ such that $t^{\infty}u \in M^{\infty}$. So
		\begin{eqnarray}
		\nonumber  J_{a_0,b_0} (t^{\infty}u)   &=&\frac{1}{2}||t^{\infty}u||_A^2-\frac{(t^{\infty})^q}{q}\int_{\R^N}a_-|u|^qdx  -\frac{(t^{\infty})^p}{p}\int_{\R^N}b_1|u|^pdx  \\
		\nonumber    &\geq&  \frac{1}{2}||t^{\infty}u||_A^2  -\frac{(t^{\infty})^p}{p}\int_{\R^N}|u|^pdx \\
		\nonumber   &=&  J_{\infty} (t^{\infty}u) \geq m_{\infty},
		\end{eqnarray}
		therefore
		\begin{equation}\label{leq}
		\inf_{u \in M_{a_0,b_0}} J_{a_0,b_0} (t^{\infty}u)\geq m_{\infty}.
		\end{equation}
		By (\ref{geq}) and (\ref{leq})
		$$\inf_{u \in M_{a_0,b_0}}J_{a_0,b_0}(u)=\inf_{u \in M^{\infty}}J_{\infty}(u)=m^{\infty}.$$
	\end{proof}

	To prove our result we will need this lemma that establishes values of $\lambda$ and $\mu$ suitable values to get the fourth solution of the problem.
	
	\begin{lemma}\label{Lemma5.5}
		Exist $\lambda_0>0$ and $\mu_0>0$ with
		$$
		\lambda_0^{p-2}(1+\mu_0||b_1||_{\infty})^{2-q}<\left( \frac{q}{2}\right)^{p-2}\Upsilon_0,
		$$
		such that for all $\lambda \in (0,\lambda_0)$ and $\mu \in (0,\mu_0)$, we have
		$$\int_{\R^N}\frac{x}{|x|}(|\nabla u|^2 + u^2)dx \neq 0$$
		for all $u\in M^-_{a_{\lambda},b_{\mu}}$ with $ \jf(u)< m^+_{a_{\lambda},b_{\mu}}+m^{\infty}$.
	\end{lemma}
	
	\begin{proof} The proof is in accordance with what was done in \cite[Lemma 7.6]{doc1}.
	\end{proof}

	\section{ Fourth Solution}
	We will work in this section with estimates of the energy levels of the functional associated with the main problem to prove the existence of a solution whose energy level satisfies the conditions of Proposition \ref{Lemmadecompacidade}(ii), that is, a distinct solution of the three solutions already found in previous sections.
	For $ \alpha >0$, we define
	$$ J_{0,\alpha b_0}(u)  =  \frac{1}{2}\int_{\R^N} 	|\nabla_A u|^2 +u^2 dx  -  \frac{1}{p}\int_{\R^N} 	\alpha b_0|u|^p dx,$$
	
	$$ M_{0,\alpha b_0} = \{  u\in \HA \setminus\{0\} ; \langle J'_{0,\alpha b_0}(u),u  \rangle = 0\}.$$
	We now define the following subset of unitary ball
	$$ \mathcal{B} = \{   u\in \HA \setminus\{0\} ; u\geq 0 \;\; \mbox{e}\;\; ||u||_{A} = 1 \}.$$
	Let us recall that for every  $ u \in \HA \setminus\{0\}$   there exists a unique $t^-(u) > 0$ and $ t_0(u) > 0$ such that $t^-(u) \in M^-_{a_{\lambda},b_{\mu}}$ and $ t_0(u)  \in   M_{0, b_0}$. In order to apply the minimax argument of Bahri-Li we present the following result.
	
	\begin{lemma}\label{Lemma7.1}
		
		For each $ u  \in  \mathcal{B}$ we will have
		
		\begin{description}
			\item[$(i)$   ]    There is a single $t^{\alpha}_0 =   t^{\alpha}_0 (u) > 0$ such that $  t^{\alpha}_0 u  \in   M_{0, \alpha b_0}$ and
			$$\sup_{t\geq 0}  J_{0,\alpha b_0}   (tu)  =   J_{0,\alpha b_0} (t^{\alpha}_0u)  = \frac{  p  - 2}{2p}\left( \int_{ \R^N }\alpha b_0 |u|^p dx \right)^\frac{  - 2}{p - 2}.$$
			\item[ $(ii)$  ]     For $  \rho   \in  (0, 1),$
			
			$$ J_{ a_{\lambda}, b_{\mu}}(t^- (u)u) \geq \frac{ (1  -  \rho )^{\frac{p}{p - 2}}}{(1+  \mu  || b_2/b_1 ||_{  \infty} )^{\frac{2}{p - 2}}} J_{0, b_0} (  t_0 (u)u ) - \frac{ 2  - q}{2q} ( \rho  S_p )^{\frac{q}{q - 2}}( \lambda  ||a_+||_{q^* })^{\frac{ 2}{2 - q}}$$
			and
			$$J_{ a_{\lambda}, b_{\mu}}(t^- (u)u)  \leq  \frac{(1 + \rho )^{\frac{p}{p-2}}}{2}J_{0, b_0}( t_0 (u)u)+\frac{ 2 -  q}{2q}( \rho S_p)^{\frac{q}{q - 2}}( \lambda  ||  a_+  ||_{ q^*} + || a_-  ||_{ q^*}) ^{\frac{ 2}{2 - q}}.$$
		\end{description}
		
	\end{lemma}
	
	\begin{proof}
		\begin{description}
			\item[	$(i)$   ]  For each $ u  \in   \mathcal{B}$, consider
		\end{description}	
		$$ K_u(t)  =   J_{0,\alpha b_0}   (tu)  = \frac{ 1}{2}t^2  -  \frac{ 1}{2}t^p\int_{ \R^N }\alpha b_0 |u|^p dx,$$
		so $ K_u(t)  \rightarrow   -  \infty$  as $ t  \rightarrow   \infty$ and
		
		$$K_u'(t)  =  t  -  t^{p - 1} \int_{ \R^N }\alpha b_0 |u|^p dx.$$
		Thus, $ K_u'( t^{\alpha}_0 )  =  0,$ and $  t^{\alpha}_0u  \in   M_{0, \alpha b_0}$ as
		
		$$ t^{\alpha}_0=t^{\alpha}_0(u)  = \left( \int_{ \R^N }\alpha b_0 |u|^p dx\right)^{ \frac{1}{2 - p}}> 0.$$
		Moreover, $ K_u''(t)=  1  -  (p  -1)t^{p - 2} \int_{ \R^N }\alpha b_0 |u|^p dx.$ So, in $t^{\alpha}_0(u)$ we have
		$$ K_u''( t^{\alpha}_0 )  =  2  - p <0,$$
		that is, $t^{\alpha}_0$ is a maximum point of $K_u$.
		Then, there exists a unique $  t^{\alpha}_0=t^{\alpha}_0(u)> 0 $ such that $t^{\alpha}_0 u  \in   M_{0, \alpha b_0} $ and also by definition $ K_u (t) = J (tu) $ we get	
		$$\sup_{t \geq 0}  J_{0,\alpha b_0}   (tu)  =   J_{0,\alpha b_0}(t^{\alpha}_0u) = \frac{ p  -  2}{2p}  \left( \int_{ \R^N } \alpha b_0 |u|^p dx\right)^{ \frac{-2}{2 - p}}.$$

		\begin{description}
			\item[	$(ii)$] Consider $ \alpha   =  (1+  \mu  || b_2/b_1 ||_{  \infty} )/(1  -  \rho ).$ Then, for each $ u  \in  \mathcal{B}$ and $  \rho   \in  (0, 1),$ we have   
		\end{description} 
		\begin{eqnarray}\label{7.1}
		\nonumber    \int_{ \R^N }a_{\lambda}| t^{\alpha}_0u |^q dx  &\leq &  \lambda S_p^{\frac{-q}{2}}  ||  a_+  ||_{q^*}|| t^{\alpha}_0u ||^q_{A}  \\
		\nonumber    &\leq& \frac{ 2  - q}{2}\left( ( \rho  S_p )^{\frac{-q}{ 2}} \lambda  ||  a_+  ||_{q^* }\right)^{\frac{ 2}{2 - q}}+\frac{q}{ 2}\left( ( \rho   )^{\frac{q}{ 2}}   ||t^{\alpha}_0u ||_A\right)^{\frac{ 2}{q} } \\
		&=&\frac{ 2  - q}{2} ( \rho  S_p )^{\frac{q}{ q-2}}( \lambda  ||  a_+  ||_{q^* })^{\frac{ 2}{2 - q}}+\frac{q \rho}{ 2} ||t^{\alpha}_0u ||_A^2.
		\end{eqnarray}
		Then, for the part $ (i) $ and by (\ref{7.1}),	
		\begin{eqnarray}
		\nonumber   \sup_{t\geq0} J_{ a_{\lambda}, b_{\mu}}(tu) &\geq &  J_{ a_{\lambda}, b_{\mu}}(t^{\alpha}_0 u) \\
		\nonumber    &\geq & \frac{1  -  \rho}{2}|| t^{\alpha}_0u||^2_{A}-\frac{  2  - q}{2q}( \rho  S_p )^{\frac{q}{q-2}}( \lambda  ||  a_+  ||_{ q^*})^{\frac{2}{2 - q} } \\
		\nonumber    &&-\frac{(1+  \mu  || b_2/b_1 ||_{  \infty} )}{p}\int_{ \R^N } b_0 |t^{\alpha}_0u|^p dx\\
		\nonumber    &=&(1  -  \rho)J_{0,\alpha b_0}(t^{\alpha}_0u) -\frac{  2  - q}{2q}( \rho  S_p )^{\frac{q}{q-2}}( \lambda  ||  a_+  ||_{q^*})^{\frac{2}{2 - q}}\\
		\nonumber    &=&\frac{(p-2)(1  -  \rho)^{\frac{p}{p-2}}}{2p((1+  \mu  || b_2/b_1 ||_{  \infty} )\int_{ \R^N } b_0 |u|^p dx )^{\frac{2}{p-2}}} -\frac{  2  - q}{2q}( \rho  S_p )^{\frac{q}{q-2}}( \lambda  ||  a_+  ||_{ q^*})^{\frac{2}{2 - q}}\\
		\nonumber    &=&\frac{(1  -  \rho)^{\frac{p}{p-2}}}{(1+  \mu  || b_2/b_1 ||_{  \infty} )^{\frac{2}{p-2}}}J_{0,\alpha b_0}(t_0(u)u) -\frac{  2  - q}{2q}( \rho  S_p )^{\frac{q}{q-2}}( \lambda  ||  a_+  ||_{ q^*})^{\frac{2}{2 - q}}.
		\end{eqnarray}
		Still, by Lemma \ref{Nehari} and by Theorem \ref{teorema3.1},
		$$\sup_{t\geq0} J_{ a_{\lambda}, b_{\mu}}(tu) =   J_{ a_{\lambda}, b_{\mu}}(t^-(u) u).$$
		Thus,
		$$ J_{ a_{\lambda}, b_{\mu}}(t^-(u)u) \geq  \frac{(1  -  \rho)^{\frac{p}{p-2}}}{(1+  \mu  || b_2/b_1 ||_{  \infty})^{\frac{2}{p-2}}}J_{0,\alpha b_0}(t_0(u)u) -\frac{  2  - q}{2q}( \rho  S_p )^{\frac{q}{q-2}}(\lambda  ||  a_+  ||_{ q^*})^{\frac{2}{2 - q}}.$$
		Further, by H\"{o}lder, Sobolev and Young's inequalities
		\begin{eqnarray}
		\nonumber    \left| \int_{ \R^N} a_{\lambda}|tu|^q dx\right|  &\leq&\int_{ \R^N} a_{\lambda}|tu|^q dx \leq  ( \lambda  ||  a_+  ||_{ q^*} + || a_-  ||_{ q^*}) S_p^{\frac{-q}{2}}|| tu ||^q_{A} \\
		\nonumber    &\leq & \frac{  2  - q}{2}( \rho  S_p )^{\frac{q}{q-2}}( \lambda  ||  a_+  ||_{ q^*} + || a_-  ||_{ q^*})^{\frac{2}{2 - q}} + \frac{q\rho}{2}||tu||^2_{A}.
		\end{eqnarray}
		Also,
		\begin{eqnarray}
		\nonumber    J_{a_{\lambda}, b_{\mu}}(tu) & \leq &   \frac{ (1 +  \rho )}{2}t^2 +  \frac{  2  - q}{2q}( \rho  S_p)^{\frac{q}{q-2}}( \lambda  ||  a_+  ||_{ q^*} + || a_-  ||_{ q^*})^{\frac{2}{2 - q}}- \frac{1}{p}\int_{ \R^N} b_0 |tu|^p dx\\
		\nonumber  &\leq & \frac{(1 +  \rho )^{\frac{p}{p-2}}}{2} J_{0, b_0}( t_0 (u)u) +  \frac{  2  - q}{2q}( \rho S_p)^{\frac{q}{q-2}}( \lambda  ||  a_+  ||_{ q^*} + || a_-  ||_{ q^*})^{\frac{2}{2 - q}}.
		\end{eqnarray}
		Then,
		$$J_{ a_{\lambda}, b_{\mu}}(t^-(u)u) \leq \frac{ (1 + \rho )^{\frac{p}{p-2}}}{2}J_{0, b_0}( t_0 (u)u) +  \frac{  2  - q}{2q}( \rho  S_p )^{\frac{q}{q-2}}( \lambda  ||  a_+  ||_{ q^*} + || a_-  ||_{ q^*})^{\frac{2}{2 - q}}.$$
		As we wanted to prove.
		
	\end{proof}

	Note that as $m^-_{a_{\lambda}, b_{\mu} }>0$ for all $\lambda\in (0,\lambda_0)$ and $\mu \in (0,\mu_0)$, we define
	$$I_{a_{\lambda}, b_{\mu} }(u)= \sup_{t \geq 0}  J_{a_{\lambda}, b_{\mu} }   (tu)  =   J_{a_{\lambda}, b_{\mu} }(t^-(u)u)>0,$$
	where $t^-(u)u \in M_{a_{\lambda}, b_{\mu} }^- $. We can see that if  $\lambda, \mu$ and $||a_-||_{ q^*}$ are sufficiently small, we can use the minimax Bahri-Li's argument \cite{BahriLi} for our functional  $J_{ a_{\lambda}, b_{\mu}} $. Let
	
	$$\Gamma_{a_{\lambda}, b_{\mu}} =\{  \gamma \in C( \overline{B^N(0,k)},\mathbb{B}); \gamma|_{\partial B^N(0,k)}=w_k/||w_k||_{A}\}$$
	be for values of $ l $ large enough.
	
	We define
	$$ n_{a_{\lambda}, b_{\mu}} =\inf_{\gamma \in \Gamma_{a_{\lambda}, b_{\mu}}} \sup_{x\in\R^N} I_{a_{\lambda}, b_{\mu}}(\gamma(x)) \;\;     \mbox{e}   $$
	
	$$n_{0, b_{0}} =\inf_{\gamma \in \Gamma_{0, b_{0}}} \sup_{x\in\R^N} I_{0, b_{0}}(\gamma(x))  $$
	By Lemma \ref{Lemma7.1}(ii), for $0<\rho<1,$ we have
	\begin{equation}\label{7.2}
	n_{a_{\lambda}, b_{\mu}}\geq  \frac{(1  -  \rho)^{\frac{p}{p-2}}}{(1+  \mu  || b_2/b_1 ||_{  \infty})^{\frac{2}{p-2}}}n_{0, b_0} -\frac{  2  - q}{2q}( \rho  S_p )^{\frac{q}{q-2}}(\lambda  ||  a_+  ||_{ q^*})^{\frac{2}{2 - q}}
	\end{equation}
	and
	\begin{equation}\label{7.3}
	n_{a_{\lambda}, b_{\mu}}\leq  (1  +  \rho)^{\frac{p}{p-2}}n_{0, b_0} + \frac{  2  - q}{2q}( \rho  S_p )^{\frac{q}{q-2}}(\lambda  ||  a_+  ||_{ q^*} + ||  a_-  ||_{ q^*}   )^{\frac{2}{2 - q}}.
	\end{equation}
	We will need estimates of energy levels as follows.
	
	\begin{lemma}
		$m^{\infty}< n_{0,b_0}<2m^{\infty}.$
	\end{lemma}
	\begin{proof}
		From the results of Bahri and Li \cite{BahriLi} we have that the equation $ (E_ {0, b_0}) $ admits at least one solution  $u_0$ with $J_{0,b_0}(u_0)=n_{0,b_0}< 2m^{\infty}. $ In addition, by the condition $ (C_1) $, the equation $ (E_{0, b_0}) $ does not have a minimum energy solution. Like this, $m^{\infty}< n_{0,b_0}< 2m^{\infty}$.
	\end{proof}
	
	\begin{theorem}\label{teo7.3}
		Let $ \lambda_0 $ and $ \mu_0 $ be as in Lemma \ref{Lemma5.5}. Then there will be positive values $\tilde{\lambda_0}\leq\lambda_0$,  $\tilde{\mu_0}\leq \mu_0$ and $\tilde{\nu_0}\leq \nu_0$ such that for $\lambda \in (0, \tilde{\lambda_0})$, $\mu \in (0, \tilde{\mu_0})$ and $||  a_-  ||_{ q^*}<\nu_0, $ we have	
		$$  m^+_{a_{\lambda}, b_{\mu}}+m^{\infty} < n_{a_{\lambda}, b_{\mu}} < m^-_{a_{\lambda}, b_{\mu}} +m^{\infty}.  $$
		In addition, $ (P_1) $ admits a solution $v_{a_{\lambda}, b_{\mu}}$ with
		$$J_{a_{\lambda}, b_{\mu}}(v_{a_{\lambda}, b_{\mu}})=n_{a_{\lambda}, b_{\mu}}.$$
		
	\end{theorem}
	\begin{proof}
		By Lemma \ref{Lemma7.1}(ii), we have for $0<\rho<1$
		$$  m^-_{a_{\lambda}, b_{\mu}}\geq   \frac{(1  -  \rho)^{\frac{p}{p-2}}}{(1+  \mu  || b_2/b_1 ||_{  \infty})^{\frac{2}{p-2}}}m^{\infty} -\frac{  2  - q}{2q}( \rho  S_p )^{\frac{q}{q-2}}(\lambda  ||  a_+  ||_{ q^*})^{\frac{2}{2 - q}} $$
		and	
		$$m^-_{a_{\lambda}, b_{\mu}} \leq  (1  +  \rho)^{\frac{p}{p-2}}m^{\infty} + \frac{  2  - q}{2q}( \rho S_p)^{\frac{q}{q-2}}(\lambda  ||  a_+  ||_{ q^*} + ||  a_-  ||_{ q^*}   )^{\frac{2}{2 - q}}.$$
		For each $ \epsilon> 0 $ there are positive values $\tilde{\lambda_1}\leq\lambda_0$,  $\tilde{\mu_1}\leq \mu_0$ and $\nu_1$ such that $\lambda \in (0, \tilde{\lambda_1})$, $\mu \in (0, \tilde{\mu_1})$ and $||  a_-  ||_{ q^*}<\nu_1, $ we have
		$$ m^{\infty} - \epsilon < n_{a_{\lambda}, b_{\mu}} <  m^{\infty} + \epsilon.  $$
		Then,	
		$$ 2m^{\infty} - \epsilon < n_{a_{\lambda}, b_{\mu}}+m^{\infty} <  2m^{\infty} + \epsilon.  $$
		Using \ref{7.2} and \ref{7.3}, for all $\delta>0$, there will be positive values $\tilde{\lambda_2}\leq\lambda_0$,  $\tilde{\mu_2}\leq \mu_0$ and $\nu_2$ such that for $\lambda \in (0, \tilde{\lambda_2})$, $\mu \in (0, \tilde{\mu_2})$ and $||  a_- ||_{ q^*}<\nu_2, $ we have
		
		$$ n_{0,b_0} - \delta < n_{a_{\lambda}, b_{\mu}} <  n_{0,b_0} + \delta.  $$
		Fixing small values of $0<\epsilon<( 2m^{\infty} - n_{0,b_0}  )/2$, and being $ m^{\infty}   < n_{0, b_0}<2m^{\infty}, $ and choosing $\delta>0$ so that for $ \lambda<\tilde{\lambda_0}=\min\{\tilde{\lambda_1},\tilde{\lambda_2}\}$, $ \mu < \tilde{\mu_0} = \min\{\tilde{\mu_1},\tilde{\mu_2}\}$ and $||  a_-  ||_{ q^*}<\nu_0=\min\{\nu_1,\nu_2\}$, we will have
		$$m^+_{a_{\lambda}, b_{\mu}} + m^{\infty} < m^{\infty}  < n_{a_{\lambda}, b_{\mu}}<   2m^{\infty} - \epsilon < m^-_{a_{\lambda}, b_{\mu}}+m^{\infty} .  $$
		Thus, by Proposition \ref{Lemmadecompacidade}(ii), we obtain that the problem $ \Plm $ has a solution $v_{a_{\lambda}, b_{\mu}}$ with
		$$J_{a_{\lambda}, b_{\mu}}(v_{a_{\lambda}, b_{\mu}})=n_{a_{\lambda}, b_{\mu}}.$$
		
	\end{proof}
	
	\textbf{Proof of Theorem \ref{teo1.2}:} With the result of theorem \ref{teo7.3} we can complete the proof of theorem \ref{teo1.2}. For $\lambda \in (0, \tilde{\lambda_0})$, $\mu \in (0, \tilde{\mu_0})$ and $||a_- ||_{ q^*}<\nu_0 $, also using the results presented in the introduction about the existence of the first three solutions and \ref{teo7.3}, we obtain that the equation $\Plm $ admits at least four solutions.
	
	\bibliographystyle{unsrt}  


\begin{thebibliography}{1}
		
		\bibitem{ClauFig}
		Alves, C. O. and Figueiredo, G.M. 
		\newblock Multiple Solutions for a Semilinear Elliptic Equation with Critical Growth and Magnetic Field.
		\newblock In {\em Milan J. Math}, 82.2, pages 389-405, 2014.
		
		\bibitem{AmbBre} 
		Ambrosetti,A., Brezis,H. and Cerami,G.
		\newblock Combined effects of concave and convex nonlinearities in some elliptic problems,
		\newblock In {\em J. Func. Anal.}, 122.2, pages 519-543, 1994.
		
		\bibitem{ArSz} 
		Arioli, G. and Szulkin, A. 
		\newblock A semilinear Schrödinger equation in the presence of a magnetic field.
		\newblock In {\em Archive for Rational Mechanics and Analysis}, 170.4, pages 277-295, 2003.
		
		\bibitem{BahriLi} 
		Bahri,A. and Li,Y.Y., 
		\newblock On a Min-Max Procedure for the Existence of a Positive Solution for Certain Scalar Field Equations in $\R^N$.
		\newblock In {\em Rev. Mat. Iberoamericana}, 6.1, pages 1-15, 1990.
		
		\bibitem{BerLions}
		Berestycki, H., Lions, P. L. 
		\newblock Nonlinear scalar field equations, I existence of a ground state.
		\newblock In {\em Archive for Rational Mechanics and Analysis}, 82.4, pages 313-345, 1983.
		
		\bibitem{BW} 
		Brown, K.J., and Wu,T.F. 
		\newblock A fibering map approach to a semilinear elliptic boundary value problem.
		\newblock In {\em J. Differential Equations}, v. 2007, pages 1-9, 2007.
		
		\bibitem{BZ2003} 
		Brown, K.J. and Zhang, Y.  
		\newblock The Nehari monifold for a semilinear elliptic problem with a sign changing weight function.
		\newblock In {\em J. Diferencial Equations}, 193.2, pages 481-499, 2003.
		
		\bibitem{B}
		Brown, K. J.  
		\newblock The Nehari manifold for a semilinear elliptic equation involving a sublinear term. 
		\newblock In {\em Calculus of Variations and Partial Differential Equations}, 22.4, pages 483-494, 2004.
		
		\bibitem{ChabSzul}
		Chabrowski,J. and Szulkin,A., 
		\newblock On the Schrödinger equation involving a critical Sobolev exponent and magnetic field.
		\newblock In {\em Topol. Meth. Nonl. Anal.}, 25, pages 3-21, 2005.
		
		\bibitem{Chen}
		Chen,K.J., 
		\newblock On multiple solutions of concave and convexe nonlinearities in elliptic equation on RN.
		\newblock In {\em BVP}, ID 147008, pages 1-19, 2009.
		
		\bibitem{Cingolani2} 
		Cingolani,S., Jeanjean,L. and Secchi,S. 
		\newblock Multi-peak solutions for magnetic NLS equations without non-degeneracy conditions.
		\newblock In {\em  ESAIM Control Optim. Calc. Var.}, 15.3, pages 653-675, 2009.
		
		\bibitem{doc1}
		de Paiva,F.O., de Souza Lima, S.M. and Miyagaki, O.H.  
		\newblock  Existence and regularity for a Schrödinger equation with magnetic potential.
		Submited 2018.
		
		\bibitem{DL}
		Ding,Y. and Liu,X.Y., 
		\newblock Semiclassical solutions of Schrödinger equations with magnetic fields and critical nonlinearities.
		\newblock In {\em  Manuscripta Math.}, pages 1-32, 2013.
		
		\bibitem{EstLions} 
		Esteban, M. J. and Lions, P. L. 
		\newblock Stationary solutions of nonlinear Schrödinger equations with an external magnetic field.
		\newblock In {\em PDE and Calculus of Variations, in honor of E. De Giorgi}, Boston: Birkhauser, pages 401-449, 1990.
		
		\bibitem{GNN} 
		Gidas, B. Nirenberg, 
		\newblock Symmetry of positive solutions of nonlinear elliptic equations in $\R^N$.
		\newblock In {\em  Adv. Math. Suppl. Stud.}, 7, pages 369-402, 1981.
		
		\bibitem{Hewitt} 
		Hewitt, E.  and Stromberg, K. 
		\newblock Real and Abstract Analysis.
		\newblock In {\em Berlin: Springer-Verlag}, 1975.
		
		\bibitem{Hsu} 
		Hsu,T.S. and Lin,H.L., 
		\newblock Three positive solutions for semilinear elliptic problems involving concave and convex nonlinearities.
		\newblock In {\em Proc. Roy. Soc. Edinburgh}, Sect. A 142.1, page 115, 2012.
		
		\bibitem{Hsu1} 
		Hsu,T.S 
		\newblock Multiple positive solutions for a class of concave-convex semilinear elliptic equations in unbounded domains with sign-changing weights.
		\newblock In {\em BVP} 2010.1, pages 856932, 2010.
		
		\bibitem{HWW} 
		Huang,Y., Wu,T.F. and Wu,Y., 
		\newblock Multiple positive solutions for a class of concave-convex elliptic problems in $\R^N$ involving sign-changing weight, II.
		\newblock In {\em  Commun. Contemp. Math.}, 17.05, 1450045, 2015.
		
		\bibitem{Kurata}
		Kurata, K.
		\newblock Existence and semi-classical limit of the least energy solution to a nonlinear Schrödinger equation with electromagnetic fields.
		\newblock In {\em  Nonlinear Anal.}, 41 Multiple positive solutions for a class of concave-convex semilinear elliptic equations in unbounded domains with sign-changing weights.
		\newblock In {\em BVP} 2010.1, pages 763–778, 2000.
		
		\bibitem{KWONG} 
		Kwong,M.K., 
		\newblock Uniqueness of positive solutions of $\Delta u- u + u^p = 0$ in $\R^n$.
		\newblock In {\em Archive for Rational Mechanics and Analysis}, 105.3, pages 243-266, 1989.
		
		\bibitem{Rab} 
		Rabinowitz,P.H., 
		\newblock Minimax methods in critical point theory with applications to differential equations.
		\newblock In {\em   American Mathematical Soc.},No. 65., Providence, Rhode Island, 1986.
		
		\bibitem{TZ}
		Tang, Z.W. 
		\newblock Multi-bump bound states of nonlinear Schrödinger equations with electromagnetic fields and critical frequency.
		\newblock In {\em  J. Differential Equations}, 245.10, pages 2723-2748, 2008.
		
		\bibitem{willem} 
		Willem,M., 
		\newblock Minimax Theorems.
		\newblock In {\em  Basel: Birkhäuser}, 1996.
		
		\bibitem{Wu}
		Wu, T.F.  
		\newblock Multiple positive solutions for a class of concave–convex elliptic problems in $\R^N$  involving sign-changing weight.
		\newblock In {\em  J. Functional Analysis}, 258.1, pages 99-131, 2010. 
		
		\bibitem{Wu2} 
		Wu,T.F., 
		\newblock On semilinear elliptic equations involving concave-convex nonlinearities and sign-changing weight function.
		\newblock In {\em  J. Math. Anal. Appl.}, 318, pages 253-270, 2006.
		
	\end{thebibliography}

\end{document}